\documentclass[a4paper,10pt]{article}
\usepackage[latin1]{inputenc}
\usepackage[T1]{fontenc}
\usepackage{amsmath}
\usepackage{amsfonts}
\usepackage{amstext}
\usepackage{amsthm}
\usepackage[all]{xy}
\usepackage{geometry}
\usepackage{srcltx}
\usepackage{graphics}
\usepackage{epsfig}
\usepackage{subfigure}
\usepackage{slashed}
\usepackage{color}
\usepackage{amssymb}
\usepackage{srcltx}
\newtheorem{theorem}{Theorem}[section]
\newtheorem{lemma}[theorem]{Lemma}
\newtheorem{defi}[theorem]{Definition}

\newtheorem{cor}[theorem]{Corollary}

\DeclareMathOperator{\im}{im}
\DeclareMathOperator{\ind}{ind}

\title{A Remark on Bifurcation of Fredholm Maps}
\author{Nils Waterstraat}

\begin{document}
\date{}
\maketitle

\footnotetext[1]{{\bf 2010 Mathematics Subject Classification: Primary 58E07; Secondary 58J20, 19L20}}

\begin{abstract}
\noindent We modify an argument for multiparameter bifurcation of Fredholm maps by Fitzpatrick and Pejsachowicz to strengthen results on the topology of the bifurcation set. Furthermore, we discuss an application to families of differential equations parameterised by Grassmannians. 
\end{abstract}

\section{Introduction}
Let $X$ and $Y$ be Banach spaces, $\mathcal{O}\subset X$ an open neighbourhood of $0\in X$ and $\Lambda$ a compact topological space. A \textit{family of $C^1$-Fredholm maps} is a map $f:\Lambda\times\mathcal{O}\rightarrow Y$ such that each $f_\lambda:=f(\lambda,\cdot):\mathcal{O}\rightarrow Y$ is $C^1$ and the Fr\'echet derivatives $D_uf_\lambda\in\mathcal{L}(X,Y)$ are Fredholm operators of index $0$ which depend continuously on $(\lambda,u)\in\Lambda\times \mathcal{O}$ with respect to the operator norm on the space of bounded linear operators $\mathcal{L}(X,Y)$. In what follows, we assume that $f(\lambda,0)=0$ for all $\lambda\in\Lambda$, and we say that $\lambda^\ast\in\Lambda$ is a \textit{bifurcation point} if every neighbourhood of $(\lambda^\ast,0)$ in $\Lambda\times\mathcal{O}$ contains an element $(\lambda,u)$ such that $u\neq 0$ and $f(\lambda,u)=0$. By the implicit function theorem, the linear operator 

\[L_{\lambda^\ast}:=D_0f_{\lambda^\ast}:X\rightarrow Y\]
is non-invertible if $\lambda^\ast$ is a bifurcation point, and so a non-trivial kernel of some $L_\lambda=D_0f_\lambda$ is a necessary assumption for the existence of a bifurcation point. Lots of efforts have been made to obtain sufficient criteria for the existence of bifurcation points by topological methods (cf. e.g. \cite{Ize}, \cite{Alexander}, \cite{Bartsch}). Here we focus on the articles \cite{JacoboTMNAI} and \cite{JacoboTMNAII} of Pejsachowicz, which are the state of the art of a line of research on which he has worked for at least three decades (cf. \cite{JacoboK}, \cite{FiPejsachowiczII}, \cite{JacoboTMNA0}, \cite{JacoboC}, \cite{JacoboJFPTA}). In order to explain these results briefly, let us recall that every family $\{L_\lambda\}_{\lambda\in\Lambda}$ of Fredholm operators has an \textit{index bundle} which is a $K$-theory class $\ind L$ in $\widetilde{KO}(\Lambda)$ that generalises the integral index for Fredholm operators to families. Moreover, Atiyah introduced in the sixties the $J$-group $J(\Lambda)$ and a natural epimorphism $J:\widetilde{KO}(\Lambda)\rightarrow J(\Lambda)$, which is called the generalised $J$-homomorphism. Pejsachowicz' bifurcation theorem \cite{JacoboTMNAI} reads as follows:

\begin{theorem}\label{JacoboCW}
Let $f:\Lambda\times\mathcal{O}\rightarrow Y$ be a family of $C^1$-Fredholm maps, parameterised by a connected finite CW-complex $\Lambda$, such that $f(\lambda,0)=0$ for all $\lambda\in\Lambda$. If there is some $\lambda_0\in\Lambda$ such that $L_{\lambda_0}$ is invertible and $J(\ind L)\neq 0$, then the family $f$ has at least one bifurcation point from the trivial branch $\Lambda\times\{0\}$.
\end{theorem} 
\noindent Let us point out that the non-triviality of $J(\ind L)$ can be obtained by characteristic classes, what paves the way for obtaining bifurcation invariants by methods from global analysis. Pejsachowicz has obtained striking results for bifurcation of families of nonlinear elliptic boundary value problems along these lines in \cite{JacoboTMNAI}.\\
The paper \cite{JacoboTMNAII} deals with the question if $J(\ind L)$ not only provides information about the existence of a single bifurcation point, but if we can also get results about the dimension and topology of the set of all bifurcation points $B(f)\subset \Lambda$ of $f$. Let us recall that the \textit{covering dimension} $\dim(\mathcal{X})$ of a topological space $\mathcal{X}$ is the minimal value of $n\in\mathbb{N}$ such that every finite open cover of $\mathcal{X}$ has a finite open refinement in which no point is included in more than $n+1$ elements. In what follows, we denote for $k\in\mathbb{N}$ by $w_k(\ind L)\in H^k(\Lambda;\mathbb{Z}_2)$ the Stiefel-Whitney classes and by $q_k(\ind L)\in H^{2(p-1)k}(\Lambda;\mathbb{Z}_p)$ the Wu classes of the index bundle $\ind L\in\widetilde{KO}(\Lambda)$, where we assume in the latter case that $\ind L$ is orientable. Pejsachowicz' bifurcation theorem \cite{JacoboTMNAII}, which is based on his joint paper \cite{FiPejsachowiczII} with Fitzpatrick, reads as follows:

\begin{theorem}\label{mainJacobo}
Let $\Lambda$ be a compact connected topological manifold of dimension $n\geq 2$ and let $f:\Lambda\times\mathcal{O}\rightarrow Y$ be a continuous family of $C^1$-Fredholm maps verifying $f(\lambda,0)=0$ and such that there is some $\lambda_0\in\Lambda$ for which $L_{\lambda_0}$ is invertible.

\begin{enumerate}
	\item[(i)] If $\Lambda$ and $\ind L$ are orientable and, for some odd prime $p$, there is $1\leq k\leq\frac{n-1}{2(p-1)}$ such that $q_k(\ind L)\neq 0\in H^{2(p-1)k}(\Lambda;\mathbb{Z}_p)$, then the covering dimension of the set $B(f)$ is at least $n-2(p-1)k$.
	\item[(ii)] If $w_k(\ind L)\neq 0\in H^k(\Lambda;\mathbb{Z}_2)$ for some $1\leq k\leq n-1$, then the dimension of $B(f)$ is at least $n-k$.
\end{enumerate}
Moreover, either the set $B(f)$ disconnects $\Lambda$ or it cannot be deformed in $\Lambda$ into a point.
\end{theorem}
\noindent Note that if $\ind L$ is non-orientable, and so (i) cannot be applied, then $w_1(\ind L)\neq 0$ and we obtain by (ii) the strongest result that Theorem \ref{mainJacobo} can yield. Moreover, as in \cite{JacoboTMNAI}, the characteristic classes of $\ind L$ in Theorem \ref{mainJacobo} can be computed for families of elliptic boundary value problems by methods from global analysis.\\
A little downside of the final statement on the topology of $B(f)$ in Theorem \ref{mainJacobo} is that there is a priori no knowledge about which of the alternatives hold. The aim of this short article is to show that some simple modifications in Pejsachowicz' proof of Theorem \ref{mainJacobo} show that actually $B(f)$ is never a contractible topological space. Note that this is slightly weaker than the second alternative in Theorem \ref{mainJacobo} as every contractible subspace of $\Lambda$ can also be deformed in $\Lambda$ into a point.\\
After having stated and proved our version of Theorem \ref{mainJacobo} in the following section, we discuss an application to families of boundary value problems of ordinary differential equations.


\section{Theorem and Proof}

\subsection{Statement of the Theorem}
We denote by $\Phi(X,Y)$ the subspace of $\mathcal{L}(X,Y)$ consisting of all Fredholm operators, and so the derivatives $L_\lambda=D_0f_\lambda:X\rightarrow Y$ define a family $L:\Lambda\rightarrow\Phi(X,Y)$ parameterised by the compact space $\Lambda$. Atiyah and Jänich introduced independently the \textit{index bundle} for families of Fredholm operators (cf. e.g. \cite{AtiyahK}, \cite{indbundleIch}). Accordingly, there is a finite dimensional subspace $V\subset Y$ such that

\begin{align}\label{subspace}
\im L_\lambda+V=Y,\quad\lambda\in\Lambda.
\end{align}
Hence if we denote by $P$ the projection onto $V$, then we obtain a family of exact sequences

\[X\xrightarrow{L_\lambda}Y\xrightarrow{I_Y-P}\im(I_Y-P)\rightarrow 0,\]
and so a vector bundle $E(L,V)$ consisting of the union of the kernels of the maps $(I_Y-P)\circ L_\lambda$, $\lambda\in\Lambda$. As $L$ is Fredholm of index $0$, the dimensions of $E(L,V)$ and $V$ coincide. Hence, if $\Theta(V)$ stands for the product bundle $\Lambda\times V$, then we obtain a reduced $KO$-theory class

\[\ind(L):=[E(L,V)]-[\Theta(V)]\in\widetilde{KO}(\Lambda),\]
which is called the \textit{index bundle} of $L$. Let us mention for later reference that the index bundle is natural, i.e. if $\Lambda'$ is another compact topological space and $f:\Lambda'\rightarrow\Lambda$ a continuous map, then $f^\ast L:\Lambda'\rightarrow\Phi(X,Y)$ defined by $(f^\ast L)_\lambda=L_{f(\lambda)}$ is a family of Fredholm operators parameterised by $\Lambda'$ and 

\begin{align}\label{naturality}
\ind(f^\ast L)=f^\ast\ind(L). 
\end{align}
Two vector bundles $E$, $F$ over $\Lambda$ are called \textit{fibrewise homotopy equivalent} if there is a fibre preserving homotopy equivalence between their sphere bundles $S(E)$ and $S(F)$. Moreover, $E$ and $F$ are \textit{stably fibrewise homotopy equivalent} if $E\oplus\Theta(\mathbb{R}^m)$ and $F\oplus\Theta(\mathbb{R}^n)$ are fibrewise homotopy equivalent for some non-negative integers $m,n$. The quotient of $\widetilde{KO}(\Lambda)$ by the subgroup generated of all $[E]-[F]$ where $E$ and $F$ are stably fibrewise homotopy equivalent is denoted by $J(\Lambda)$ and the quotient map $J:\widetilde{KO}(\Lambda)\rightarrow J(\Lambda)$ is called the \textit{generalised J-homomorphism} (cf. \cite{AtiyahK}). Very little is known about $J(\Lambda)$ as these groups are notoriously hard to compute.\\
Note that in order to apply Theorem \ref{JacoboCW} it is not necessary to know $J(\Lambda)$. All what is needed is a way to decide if $J(\ind L)$ is non-trivial, which can be done by \textit{spherical characteristic classes}. We denote by $w_k(E)\in H^k(\Lambda;\mathbb{Z}_2)$ the \textit{Stiefel-Whitney classes} of real vector bundles $E$ over $\Lambda$, and by $q_k(E)\in H^{2(p-1)k}(\Lambda;\mathbb{Z}_p)$ the \textit{Wu classes} of (real) orientable vector bundles $E$ over $\Lambda$, where $p$ is an odd prime (cf. \cite{MiSta}). By the naturality of characteristic classes, if $\Lambda'$ is another compact space and $f:\Lambda'\rightarrow\Lambda$ is continuous, then

\begin{align}\label{naturalitychar}
w_k(f^\ast E)=f^\ast w_k(E),\qquad q_k(f^\ast E)=f^\ast q_k(E),\quad k\in\mathbb{N}. 
\end{align}
As $w_k$ and $q_k$ are invariant under addition of trivial bundles, they are well defined on $\widetilde{KO}(\Lambda)$. Moreover, they only depend on the stable fiberwise homotopy class of the associated sphere bundles and so they even factorise through $J(\Lambda)$. In particular, they detect elements having a non-trivial image under the $J$-homomorphism, i.e. $J(\ind L)\neq 0$ if $w_k(\ind L)$ or $q_k(\ind L)$ does not vanish for some $k$.\\
The aim of this article is to prove the following modification of Theorem \ref{mainJacobo}. 

\begin{theorem}\label{main}
Let $\Lambda$ be a compact connected topological manifold of dimension $n\geq 2$ and let $f:\Lambda\times\mathcal{O}\rightarrow Y$ be a continuous family of $C^1$-Fredholm maps verifying $f(\lambda,0)=0$. We denote by $L_\lambda:=D_0f_\lambda$ the Fr\'echet derivative of $f_\lambda$ at $0\in\mathcal{O}$ and we assume that there is some $\lambda_0\in\Lambda$ for which $L_{\lambda_0}$ is invertible.

\begin{enumerate}
	\item[(i)] If $\Lambda$ and $\ind L$ are orientable and, for some odd prime $p$, there is $1\leq k\leq\frac{n-1}{2(p-1)}$ such that $q_k(\ind L)\neq 0\in H^{2(p-1)k}(\Lambda;\mathbb{Z}_p)$, then the covering dimension of the set $B(f)$ is at least $n-2(p-1)k$.
	\item[(ii)] If $w_k(\ind L)\neq 0\in H^k(\Lambda;\mathbb{Z}_2)$ for some $1\leq k\leq n-1$, then the dimension of $B(f)$ is at least $n-k$.
\end{enumerate}
Moreover, the set $B(f)$ is not contractible to a point.
\end{theorem}
\noindent Let us point out once again that the novelty in our approach to this theorem is the result that $B(f)$ is never a contractible space.


\subsection{Proof of Theorem \ref{main}}
In this section we denote by $\check{H}^k(\mathcal{X};G)$ the \v{C}ech cohomology groups of a topological space $\mathcal{X}$ with respect to the abelian coefficient group $G$. Moreover, we use without further reference the fact that $\dim(\mathcal{X})\geq k$ if $\check{H}^k(\mathcal{X};G)\neq 0$ (cf. \cite[VIII.4.A]{Hurewicz}).\\
We now split the proof of Theorem \ref{main} into two parts depending on whether or not $\Lambda\setminus B(f)$ is connected.


\subsubsection*{Part 1: Proof of Theorem \ref{main} if $\Lambda\setminus B(f)$ is not connected}
If we assume that $\Lambda\setminus B(f)$ is not connected, then the reduced homology group $\tilde{H}_0(\Lambda \setminus B(f);\mathbb{Z}_2)$ is non-trivial as its rank is one less than the number of connected components of $\Lambda\setminus B(f)$. By the long exact sequence of reduced homology (cf. \cite[\S IV.6]{Bredon})

\begin{align*}
\ldots\rightarrow H_1(\Lambda,\Lambda\setminus B(f);\mathbb{Z}_2)\rightarrow\widetilde{H}_0(\Lambda\setminus B(f);\mathbb{Z}_2)\rightarrow\widetilde{H}_0(\Lambda;\mathbb{Z}_2)=0,
\end{align*}
where we use that $\Lambda$ is connected by assumption. Hence there is a surjective map 

\[H_1(\Lambda,\Lambda\setminus B(f);\mathbb{Z}_2)\rightarrow\widetilde{H}_0(\Lambda\setminus B(f);\mathbb{Z}_2)\] showing that $H_1(\Lambda,\Lambda\setminus B(f);\mathbb{Z}_2)$ is non-trivial. As $B(f)$ is closed, we obtain by Poincar\'e-Lefschetz duality (cf. \cite[Cor. VI.8.4]{Bredon}) an isomorphism

\begin{align*}
H_1(\Lambda,\Lambda\setminus B(f);\mathbb{Z}_2)\xrightarrow{\cong} \check{H}^{n-1}(B(f);\mathbb{Z}_2),
\end{align*}
and so $\check{H}^{n-1}(B(f);\mathbb{Z}_2)\neq 0$. Consequently, as $n\geq 2$, $B(f)$ is not contractible to a point and, moreover, we obtain $\dim B(f)\geq n-1$ which is greater or equal to any of the lower bounds stated in Theorem \ref{main}.\\
Note that, as we use $\mathbb{Z}_2$ coefficients, $\Lambda$ does not need to be orientable in this part of the proof and so the argument works under either of the assumptions (i) or (ii).

\subsubsection*{Part 2: Proof of Theorem \ref{main} if $\Lambda\setminus B(f)$ is connected}
We only give the proof under the assumption (i), as the argument for (ii) is very similar and even simpler since we do not need to take orientability into account. Let us point out that this second part of the proof follows quite closely \cite{JacoboTMNAII}. A similar argument can also be found in \cite{MaciejIch}.\\
As $\mathbb{Z}_p$ is a field, the duality pairing

\begin{align*}
\langle\cdot,\cdot\rangle:H^{2(p-1)k}(\Lambda;\mathbb{Z}_p)\times H_{2(p-1)k}(\Lambda;\mathbb{Z}_p)\rightarrow\mathbb{Z}_p
\end{align*}
is non-degenerate. Hence there is $\alpha\in H_{2(p-1)k}(\Lambda;\mathbb{Z}_p)$ such that $\langle q_k(\ind L),\alpha\rangle\neq 0\in\mathbb{Z}_p$. We let $\eta\in H^{n-2(p-1)k}(\Lambda;\mathbb{Z}_p)$ be the Poincar\'e dual of $\alpha$ with respect to a fixed $\mathbb{Z}_p$-orientation of $\Lambda$. According to \cite[Cor. VI.8.4]{Bredon}, there is a commutative diagram 

\begin{align*}
\xymatrix{&\check{H}^{n-2(p-1)k}(\Lambda;\mathbb{Z}_p)\ar[r]^{i^\ast}&\check{H}^{n-2(p-1)k}(B(f);\mathbb{Z}_p)\\
H_{2(p-1)k}(\Lambda\setminus B(f);\mathbb{Z}_p)\ar[r]^{j_\ast}&H_{2(p-1)k}(\Lambda;\mathbb{Z}_p)\ar[u]\ar[r]^(.4){\pi_\ast}&H_{2(p-1)k}(\Lambda,\Lambda\setminus B(f);\mathbb{Z}_p)\ar[u]
}
\end{align*}
where the vertical arrows are isomorphisms given by Poincar\'e-Lefschetz duality and the lower horizontal sequence is part of the long exact homology sequence of the pair $(\Lambda,\Lambda\setminus B(f))$. Because of the commutativity, the class $i^\ast\eta$ is dual to $\pi_\ast\alpha$ and we now assume by contradiction the triviality of the latter one.\\
Then, by exactness, there exists $\beta\in H_{2(p-1)k}(\Lambda\setminus B(f);\mathbb{Z}_p)$ such that $\alpha=j_\ast\beta$. Moreover, since homology is compactly supported (cf. \cite[Sect. 20.4]{May}), there is a compact connected CW-complex $P$ and a map $g:P\rightarrow\Lambda\setminus B(f)$ such that $\beta=g_\ast\gamma$ for some $\gamma\in H_{2(p-1)k}(P;\mathbb{Z}_p)$.\\
Let us recall that by assumption there is some $\lambda_0\in\Lambda$ such that $L_{\lambda_0}$ is invertible. As $\lambda_0\notin B(f)$ by the implicit function theorem and $\Lambda\setminus B(f)$ is connected, we can deform $g$ such that $\lambda_0$ belongs to its image. Clearly, this does not affect the property that $\beta=g_\ast\gamma$ and so we can assume without loss of generality that $\lambda_0\in\im g$. We now set $\overline{g}=j\circ g$ and consider

\[\overline{f}:P\times X\rightarrow Y,\quad \overline{f}(p,u)=f(\overline{g}(p),u).\] 
Then $\overline{f}(p,0)=0$ for all $p\in P$ and the linearisation of $\overline{f}_p$ at $0\in X$ is $\overline{L}_p=L_{\overline{g}(p)}$. Moreover, as $\lambda_0$ is in the image of $g$, there is some $p_0\in P$ such that $\overline{L}_{p_0}$ is invertible. Clearly, $\overline{g}$ sends bifurcation points of $\overline{f}$ to bifurcation points of $f$, and as $\overline{g}(P)\cap B(f)=\emptyset$, we see that the family $\overline{f}$ has no bifurcation points. Consequently, $J(\ind \overline{L})=0\in J(P)$ by Theorem \ref{JacoboCW} showing that $q_k(\ind \overline{L})=0\in H^{2(p-1)k}(P;\mathbb{Z}_p)$. We obtain from \eqref{naturality} and \eqref{naturalitychar}

\begin{align*}
0&=\langle q_k(\ind\overline{L}),\gamma\rangle=\langle q_k(\overline{g}^\ast(\ind L),\gamma\rangle=\langle\overline{g}^\ast q_k(\ind L),\gamma\rangle=\langle g^\ast j^\ast q_k(\ind L),\gamma\rangle\\
&=\langle q_k(\ind L),j_\ast g_\ast\gamma\rangle=\langle q_k(\ind L),j_\ast\beta\rangle=\langle q_k(\ind L),\alpha\rangle
\end{align*}
which is a contradiction to the choice of $\alpha$.\\
Consequently, $\pi_\ast\alpha$ and so $i^\ast\eta\in\check{H}^{n-2(p-1)k}(B(f);\mathbb{Z}_p)$ is non-trivial. This shows that $\dim B(f)\geq n-2(p-1)k$, and moreover $B(f)$ is not contractible to a point as $n-2(p-1)k\geq 1$.


\section{An Example}
In this section we denote by $G_n(\mathbb{R}^{2n})$ the Grassmannian manifold of all $n$-dimensional subspaces of $\mathbb{R}^{2n}$ (cf. \cite[\S 5]{MiSta}). Let us recall that $G_n(\mathbb{R}^{2n})$ is compact and that for every $n\in\mathbb{N}$ there is a canonical vector bundle $\gamma^n(\mathbb{R}^{2n})$ over $G_n(\mathbb{R}^{2n})$ which is called the \textit{tautological bundle} and which has as total space

\[\{(V,u)\in G_n(\mathbb{R}^{2n})\times\mathbb{R}^{2n}:\, u\in V\}.\]
In what follows we denote by $I=[0,1]$ the compact unit interval and we let $\Lambda$ be a compact orientable manifold of dimension $m\geq 2$. We denote by $H^1(I,\mathbb{R}^n)$ the space of all absolutely continuous functions having a square integrable derivative with respect to the usual Sobolev norm

\[\|u\|^2_{H^1}:=\int^1_0{\langle u,u\rangle\, dt}+\int^1_0{\langle u',u'\rangle\, dt},\]
and we let 

\[\varphi:\Lambda\times I\times\mathbb{R}^n\rightarrow\mathbb{R}^n\]
be a continuous function, such that each $\varphi_\lambda:=\varphi(\lambda,\cdot,\cdot):I\times\mathbb{R}^n\rightarrow\mathbb{R}^n$ is smooth, all its derivatives depend continuously on $\lambda$, and $\varphi(\lambda,t,0)=0$ for all $(\lambda,t)\in \Lambda\times I$. For a continuous map $b:\Lambda\rightarrow G_n(\mathbb{R}^{2n})$, we consider the family of differential equations

\begin{equation}\label{diffequ}
\left\{
\begin{aligned}
u'(t)&=\varphi_\lambda(t,u(t)),\quad t\in[0,1]\\
(u(0)&,u(1))\in b(\lambda)
\end{aligned}
\right.
\end{equation}
and we note that $u\equiv 0$ is a solution for all $\lambda\in\Lambda$.

\begin{defi}\label{defibiff}
A parameter value $\lambda_0\in\Lambda$ is called a bifurcation point of \eqref{diffequ} if in every neighbourhood of $(\lambda_0,0)\in \Lambda\times H^1(I,\mathbb{R}^n)$ there is $(\lambda,u)$ such that $u\not\equiv0$ is a solution of \eqref{diffequ}.
\end{defi}
\noindent
In what follows we denote by $B\subset\Lambda$ the set of all bifurcation points. The linearisation of \eqref{diffequ} at $0$ is 

\begin{equation}\label{diffequlin}
\left\{
\begin{aligned}
u'(t)&=(D_0\varphi_\lambda)(t,\cdot)u(t),\quad t\in[0,1]\\
(u(0)&,u(1))\in b(\lambda),
\end{aligned}
\right.
\end{equation}
where $D_0\varphi_\lambda(t,\cdot):\mathbb{R}^n\rightarrow\mathbb{R}^n$ denotes the derivative of $\varphi_\lambda(t,u)$ with respect to $u$ at $0\in\mathbb{R}^n$.\\
We can now state our main theorem of this section.

\begin{theorem}\label{application}
If there is some $\lambda_0\in\Lambda$ such that \eqref{diffequlin} has only the trivial solution and 

\[b^\ast(w_k(\gamma^n(\mathbb{R}^{2n})))\neq 0\in H^k(\Lambda;\mathbb{Z}_2)\]
for some $1\leq k\leq m-1$, then the dimension of $B$ is at least $m-k$ and $B$ is not a contractible topological space.
\end{theorem}
\noindent
Let us note that $G_n(\mathbb{R}^{2n})$ is itself a compact manifold, which is orientable as Grassmannians $G_n(\mathbb{R}^l)$ are orientable if and only if $l$ is even (cf. \cite{Bartik}). We obtain the following corollary for $\Lambda=G_n(\mathbb{R}^{2n})$ and $b$ the identity.

\begin{cor}\label{Corollary}
If $n\geq 2$ and there is some $V_0\in G_n(\mathbb{R}^{2n})$ such that \eqref{diffequlin} has only the trivial solution, then $B$ is a non-contractible space which has at most codimension $1$ in $G_n(\mathbb{R}^{2n})$.
\end{cor}

\noindent
Let us point out that the restriction which we impose on the dimension $n$ in Corollary \ref{Corollary} is satisfied for all equations \eqref{diffequ} that are first-order reductions of second-order scalar equations.\\
In the remainder of this section we will be concerned with the proofs of Theorem \ref{application} and Corollary \ref{Corollary}. We set 

\[\mathcal{H}_\lambda:=\{u\in H^1(I,\mathbb{R}^n):\, (u(0),u(1))\in b(\lambda)\}\]
and we note that each map

\[f_\lambda:\mathcal{H}_\lambda\rightarrow L^2(I,\mathbb{R}^n),\quad f_\lambda(u)=u'-\varphi_\lambda(\cdot,u)\]
is $C^1$. Moreover,

\[(D_vf_\lambda)u=u'-(D_v\varphi_\lambda)u\]   
and we leave it to the reader to check that each $D_vf_\lambda:\mathcal{H}_\lambda\rightarrow L^2(I,\mathbb{R}^n)$ is a Fredholm operator of index $0$.\\
Note that we cannot apply Theorem \ref{main} to the operators $f_\lambda$ as they are not defined on a single Banach space. To overcome this technical issue, we need the following lemma.

\begin{lemma}
The set

\[\mathcal{H}:=\{(\lambda,u)\in\Lambda\times H^1(I,\mathbb{R}^n):\, u\in\mathcal{H}_\lambda\}\]
is a Hilbert subbundle of the product bundle $\Lambda\times H^1(I,\mathbb{R}^n)$.
\end{lemma}

\begin{proof}
Let $\widetilde{P}:G_n(\mathbb{R}^{2n})\times\mathbb{R}^{2n}\rightarrow\mathbb{R}^{2n}$ be a family of projections in $\mathbb{R}^{2n}$ such that $\im(\widetilde{P}_V)$ is the fiber of $\gamma^n(\mathbb{R}^{2n})$ over $V$ for $V\in G_n(\mathbb{R}^{2n})$. We consider the family 

\[P:\Lambda\times H^1(I,\mathbb{R}^n)\rightarrow H^1(I,\mathbb{R}^n)\]
defined by

\[(P_\lambda u)(t)=u(t)-(1-t)P_1(I_{\mathbb{R}^{2n}}-\widetilde{P}_{b(\lambda)})(u(0),u(1))-tP_2(I_{\mathbb{R}^{2n}}-\widetilde{P}_{b(\lambda)})(u(0),u(1)),\]
where $P_1,P_2:\mathbb{R}^{2n}\rightarrow\mathbb{R}^n$ are the projections on the first and last $n$ components in $\mathbb{R}^{2n}$, respectively, and $I_{\mathbb{R}^{2n}}$ is the identity in $\mathbb{R}^{2n}$. The reader can easily check that $P^2_\lambda=P_\lambda$ and $\im(P_\lambda)=\mathcal{H}_\lambda$, i.e. $P_\lambda$ is a projection in $H^1(I,\mathbb{R}^n)$ onto $\mathcal{H}_\lambda$. Hence

\[\{(\lambda,u)\in\Lambda\times H^1(I,\mathbb{R}^{2n}):\, u\in\im(P_\lambda)\}=\mathcal{H}\]
is a subbundle of $\Lambda\times H^1(I,\mathbb{R}^{2n})$ by \cite[\S III.3]{Lang}.
\end{proof}

\noindent
As every Hilbert bundle over a finite $CW$-complex is trivial by Kuiper's Theorem and \cite[p. 54]{Steenrod}, there exists a bundle isomorphism $\psi:\Lambda\times H\rightarrow\mathcal{H}$ for some Hilbert space $H$. Moreover, by \cite[Thm. VII.3.1]{Lang} we can assume without loss of generality that each $\psi_\lambda:H\rightarrow\mathcal{H}_\lambda$ is orthogonal and so an isometry. Composing $\psi$ and $f$, we obtain a family of $C^1$ maps

\[\tilde{f}:=f\circ\psi:\Lambda\times H\rightarrow L^2(I,\mathbb{R}^n)\]
such that $f(\lambda,0)=0$ for all $\lambda\in\Lambda$, and each $\tilde{f}_\lambda:H\rightarrow L^2(I,\mathbb{R}^n)$ is Fredholm of index $0$. Consequently, we can now apply Theorem \ref{main} to the family $\tilde{f}$ and so we need to compute $\ind(\widetilde{L})\in\widetilde{KO}(\Lambda)$, where $\widetilde{L}_\lambda=D_0\widetilde{f}_\lambda$, $\lambda\in\Lambda$. By the chain rule, we get

\[\widetilde{L}_\lambda=(D_{\psi_\lambda(0)}f_\lambda)(D_0\psi_\lambda)=(D_0f_\lambda)\psi_\lambda.\]
Consequently, if we introduce two Hilbert bundle morphisms by

\[M:\Lambda\times H\rightarrow\mathcal{H},\quad M_\lambda u=\psi_\lambda u\]
and

\[L:\mathcal{H}\rightarrow\Lambda\times L^2(I,\mathbb{R}^n),\quad L_\lambda u=(D_0f_\lambda) u=u'-(D_0\varphi_{\lambda})u,\] 
then $\widetilde{L}_\lambda=L_\lambda M_\lambda$, $\lambda\in\Lambda$.\\
The construction of the index bundle, which we recalled in Section 2.1, generalises almost verbatim to Fredholm morphisms of Banach bundles (cf. \cite{indbundleIch}). Accordingly, if $\mathcal{X}$, $\mathcal{Y}$ are Banach bundles over a compact base $\Lambda$, and $A:\mathcal{X}\rightarrow\mathcal{Y}$ is a bundle morphism such that every $A_\lambda:\mathcal{X}_\lambda\rightarrow\mathcal{Y}_\lambda$ is Fredholm, then 

\[\ind(A)=[E(A,\mathcal{V})]-[\mathcal{V}]\in\widetilde{KO}(\Lambda),\] 
where $\mathcal{V}\subset\mathcal{Y}$ is a finite dimensional subbundle such that

\begin{align}\label{subbundle}
\im(A_\lambda)+\mathcal{V}_\lambda=\mathcal{Y}_\lambda,\quad \lambda\in\Lambda.
\end{align}
As the index bundle is additive under compositions of bundle morphisms, and it is trivial for bundle isomorphisms, we obtain

\begin{align}\label{indequI}
\ind(\widetilde{L})=\ind(LM)=\ind(L)+\ind(M)=\ind(L)\in\widetilde{KO}(\Lambda).
\end{align}
Let us now consider the Fredholm morphism

\[\hat{L}:\mathcal{H}\rightarrow \Lambda\times L^2(I,\mathbb{R}^n),\quad \hat{L}_\lambda u=u'.\]
Then $L_\lambda-\hat{L}_\lambda:\mathcal{H}_\lambda\rightarrow L^2(I,\mathbb{R}^n)$ is compact for all $\lambda\in\Lambda$ and so 

\begin{align}\label{indequII}
\ind(L)=\ind(\hat{L})\in\widetilde{KO}(\Lambda)
\end{align}
by \cite[Cor. 7]{indbundleIch}.\\
Our next aim is to compute $\ind(\hat{L})$. Let us denote by $Y_1$ the $n$-dimensional subspace of constant functions in $L^2(I,\mathbb{R}^{n})$ and

\[Y_2:=\left\{u\in L^2(I,\mathbb{R}^n):\, \int^1_0{u(t)\, dt}=0\right\}.\] 
Clearly, $Y_1\cap Y_2=\{0\}$, and as every function $u\in L^2(I,\mathbb{R}^n)$ can be written as 

\[u(t)=(u(t)-\int^1_0{u(s)\,ds})+\int^1_0{u(s)\,ds},\]
we see that $L^2(I,\mathbb{R}^n)=Y_1\oplus Y_2$. If $v\in Y_2$, then

\[u(t)=\int^t_0{v(s)\,ds},\quad t\in I,\] 
is an element of $\mathcal{H}_\lambda$ which is mapped to $v$. Consequently, $\im(\hat{L}_\lambda)+Y_1=L^2(I,\mathbb{R}^n)$ for all $\lambda\in\Lambda$ and we get that

\begin{align*}
E(\hat{L},Y_1)&=\{(\lambda,u)\in \Lambda\times H^1(I,\mathbb{R}^n): u\in\mathcal{H}_\lambda,\,\hat{L}_\lambda u\in Y_1\}\\
&=\{(\lambda,u)\in\Lambda\times H^1(I,\mathbb{R}^n): u(t)=(1-t)a+tb,\, (a,b)\in b(\lambda)\}\\
&\cong\{(\lambda,a,b)\in\Lambda\times \mathbb{R}^{2n}:  (a,b)\in b(\lambda)\}.  
\end{align*}
As this vector bundle is isomorphic to the pullback $b^\ast(\gamma^n(\mathbb{R}^{2n}))$ of the tautological bundle $\gamma^n(\mathbb{R}^{2n})$ over $G_n(\mathbb{R}^{2n})$ by $b$, we finally obtain by \eqref{indequI} and \eqref{indequII}

\begin{align*}
\ind(\widetilde{L})&=\ind(L)=\ind(\hat L)=[E(\hat{L},Y_1)]-[Y_1]\\
&=[b^\ast(\gamma^n(\mathbb{R}^{2n}))]-[\Theta^n]\in\widetilde{KO}(\Lambda),
\end{align*}
where $\Theta^n$ denotes the product bundle with fiber $\mathbb{R}^n$ over $\Lambda$. Hence,

\[w_k(\ind(\widetilde{L}))=w_k(b^\ast(\gamma^n(\mathbb{R}^{2n})))=b^\ast w_k(\gamma^n(\mathbb{R}^{2n}))\neq 0\in H^k(\Lambda;\mathbb{Z}_2)\]
by the assumption of Theorem \ref{application}. Moreover, as \eqref{diffequlin} has only the trivial solution for $\lambda=\lambda_0$, we see that the operator $L_{\lambda_0}:\mathcal{H}_{\lambda_0}\rightarrow L^2(I,\mathbb{R}^n)$ is injective and so invertible as it is Fredholm of index $0$. Since $\psi_{\lambda_0}:H\rightarrow\mathcal{H}_{\lambda_0}$ is invertible, we conclude that $\widetilde{L}_{\lambda_0}$ is invertible as well. Consequently, we obtain from Theorem \ref{main} that the family of $C^1$-Fredholm maps $\widetilde{f}$ has a bifurcation point $\lambda^\ast\in\Lambda$, i.e. there is a sequence $\{(\lambda_n,u_n)\}_{n\in\mathbb{N}}\subset\Lambda\times H$ converging to $(\lambda^\ast,0)$ such that $u_n\neq 0$ and $\widetilde{f}_{\lambda_n}(u_n)=f_{\lambda_n}(\psi_{\lambda_n}u_n)=0$ for all $n\in\mathbb{N}$. We now set $v_n:=\psi_{\lambda_n}u_n\neq 0\in\mathcal{H}_{\lambda_n}$ for $n\in\mathbb{N}$. As $\psi_{\lambda_n}:H\rightarrow\mathcal{H}_{\lambda_n}$ is an isometry, it follows that $\|v_n\|_{H^1}=\|u_n\|_{H}\rightarrow 0$ as $n\rightarrow\infty$, where $\|\cdot\|_H$ is the norm of $H$. Since $v_n\neq 0$ is a solution of \eqref{diffequ} by the definition of $f$, we obtain that $\lambda^\ast$ is a bifurcation point of \eqref{diffequ} in the sense of Definition \ref{defibiff}, and so Theorem \ref{application} is proved.\\
Let us now focus on the proof of Corollary \ref{Corollary}. By Theorem \ref{main} we only need to recall the well known fact that $w_1(\gamma^n(\mathbb{R}^{2n}))\neq 0\in H^1(G_n(\mathbb{R}^{2n});\mathbb{Z}_2)$. Indeed, if we consider the tautological bundle $\gamma^n(\mathbb{R}^\infty)$ over the infinite Grassmannian $G_n(\mathbb{R}^\infty)$, then by \cite[Thm. 5.6]{MiSta} for every $n$-dimensional vector bundle $\xi$ there is a bundle map $g:\xi\rightarrow\gamma^n(\mathbb{R}^\infty)$. Hence $w_1(\xi)=\overline{g}^\ast w_1(\gamma^n(\mathbb{R}^\infty))$, where $\overline{g}$ is the map between the base space of $\xi$ and $G_n(\mathbb{R}^\infty)$ that is induced by $g$. Hence if $w_1(\gamma^n(\mathbb{R}^\infty))=0$, then $w_1(\xi)$ would be trivial for every $n$-dimensional bundle $\xi$. However, this is wrong as if we set $\xi=\gamma^1(\mathbb{R}^2)\oplus\Theta^{n-1}$, where $\Theta^{n-1}$ is the product bundle with fiber $\mathbb{R}^{n-1}$ over $G_1(\mathbb{R}^2)$, then 

\[w_1(\xi)=w_1(\gamma^1(\mathbb{R}^{2}))+w_1(\Theta^{n-1})=w_1(\gamma^1(\mathbb{R}^{2})),\]
and the latter class is non-trivial as $\gamma^1(\mathbb{R}^2)$ is isomorphic to the Möbius bundle over $S^1$, which is non-orientable and so has a non-trivial first Stiefel-Whitney class.\\
If we now let $\iota:G_n(\mathbb{R}^{2n})\hookrightarrow G_n(\mathbb{R}^\infty)$ be the canonical inclusion map, then $\iota^\ast(\gamma^n(\mathbb{R}^\infty))$ is isomorphic to $\gamma^n(\mathbb{R}^{2n})$ and we obtain    

\[\iota^\ast w_1(\gamma^n(\mathbb{R}^\infty))=w_1(\iota^\ast(\gamma^n(\mathbb{R}^\infty)))=w_1(\gamma^n(\mathbb{R}^{2n}))\in H^1(G_n(\mathbb{R}^{2n});\mathbb{Z}_2).\] 
As $\iota^\ast:H^1(G_n(\mathbb{R}^\infty);\mathbb{Z}_2)\rightarrow H^1(G_n(\mathbb{R}^{2n});\mathbb{Z}_2)$ is an isomorphism for $n\geq 2$ by \cite[Problem 6-B]{MiSta}, we conclude that $w_1(\gamma^n(\mathbb{R}^{2n}))$ is non-trivial, and so Corollary \ref{Corollary} is proved.


\newpage

\vspace{1cm}
Nils Waterstraat\\
School of Mathematics,\\
Statistics \& Actuarial Science\\
University of Kent\\
Canterbury\\
Kent CT2 7NF\\
UNITED KINGDOM\\
E-mail: n.waterstraat@kent.ac.uk

\end{document}